\documentclass[1p,11pt]{article}

\usepackage{lineno,hyperref}
\usepackage{amsmath}
\usepackage{amsthm}
\usepackage{authblk}
\usepackage{graphicx,amssymb}
\usepackage{amssymb,tikz}
\usepackage{booktabs}
\usepackage{multirow}
\usepackage[normalem]{ulem}

\modulolinenumbers[5]
\newtheorem{theorem}{Theorem}
\newtheorem{exmp}{Example}

\newtheorem*{remark}{Remarks.}

\usepackage{algorithm}

\title{ Cauchy-type identities through collocation matrices}

\author[1]{P. D\'iaz\thanks{pablodiaz@unizar.es}}

\author[2]{E. Mainar\thanks{esmemain@unizar.es}}

\affil[1,2]{\textit{\small{Departamento de Matem\'{a}tica Aplicada/IUMA, Universidad de Zaragoza}}}
\date{}
\begin{document}
\maketitle
\begin{abstract}
 We present a broader framework for the Cauchy identity derived from the determinant expansion of collocation matrices. This approach yields an infinite family of identities,    where the original Cauchy identity stands as a particular case. To illustrate the versatility and depth of this approach, we provide a range of compelling examples, showcasing the connections and applications of these novel identities.
\end{abstract}



\section{Introduction} \label{Sec:Intro}

The celebrated Cauchy identity is expressed as:
\[
\prod_{i,j} \frac{1}{1 - x_i y_j} = \sum_{\lambda} s_\lambda(x) s_\lambda(y),
\]
where \( s_\lambda(x) \) and \( s_\lambda(y) \)   denote Schur functions indexed by  the partition \( \lambda \), and \( x = (x_1, x_2, \ldots) \) and \( y = (y_1, y_2, \ldots) \) are sets of variables.
This identity reveals deep connections between seemingly disparate entities, such as determinants, symmetric polynomials, and partition functions. Rigorous algebraic proofs
of the identity can be found in classical references such as \cite{Mac} and \cite{Stanley}.

Among its many interpretations, one of the most illuminating stems from representation theory. 
The Cauchy identity can be understood as a result on the characters of the General Linear group (see, for instance, \cite{Bump}, chapter 38), which makes it a fundamental piece in algebraic combinatorics.

The Cauchy identity has inspired numerous generalizations. Extensions exist for characters of Orthogonal and Symplectic groups \cite{ortho}, double-Schur polynomials \cite{double}, Hall-Littlewood polynomials \cite{Hall}, and
$q$-deformations \cite{quantum}, among others. Each of these adaptations reveals new layers of insight into the interplay between algebraic structures.

In this paper, we propose a unified framework that encompasses the Cauchy identity and its variations. 
By analyzing the determinant expansions of a family of collocation matrices, we derive an infinite class of identities, with the classical Cauchy identity emerging as a special case.

\section{Collocation matrices and Cauchy-type identities}
\begin{theorem}\label{Theorem}
    Let $g\in C^\infty$, and call $R$ the radius of convergence of the McLaurin series of $g$. Consider the $n$-tuplas $X=(x_1,\dots, x_n)$ and  $A=(a_1,\dots, a_n)$, such that $x_i\ne x_j$ and $a_i\ne a_j$ for $i\ne j$.  Let \( G = (g_1, \dots, g_n) \), where $g_k(x):=g(a_k x)$, defined on $x\in (-R/a_k,R/a_k)$. Let us call $M_G(X)$ the collocation matrix of the system $G$ on the $n$-tupla $X$,
    \[
   M_{G}(X) := \big( g(a_jx_{i}) \big)_{1 \le i, j \le n}.
   \]
For  $x_k\in(-R/a_{max},R/a_{max})$, where $a_{max}=\max\{|a_1|,\dots,|a_n|\}$, we have
\begin{equation}\label{Cor}
 \sum_{\lambda}\frac{G_\lambda}{C_\lambda}s_{\lambda}(a_1,\dots,a_n)s_{\lambda}(x_{1},\dots,x_n)=\frac{\det M_G(X)}{|V_{x_{1},\dots,x_n}||V_{a_{1},\dots,a_n}|} , 
\end{equation}
where,   for any partition $\lambda=(\lambda_1,  \ldots, \lambda_n)$, $s_{\lambda}(x_{1},\dots,x_n)$ is the corresponding Schur polynomial, and
\[
    G_\lambda:=\prod_{l=1}^n g^{(\lambda_l+n-l)}(0), \qquad  C_\lambda:=\prod_{l=1}^n(\lambda_l+n-l)!.
\]
\end{theorem}
\begin{proof}
For $x_k\in(-R/a_{max},R/a_{max})$ we can identify   $g(x)$ with its McLaurin series. Thus,
\begin{eqnarray*}
    &&\det M_G(X)=\det\big(g(a_jx_i)\big)=\det\Big(\sum_{k=0}^\infty \frac{g^{(k)}(0)}{k!}(a_jx_i)^k\Big)\nonumber \\ &\stackrel{^{(1)}}{=}&\sum_{k_1,\dots,k_n=0}^\infty \det\Big(\frac{g^{(k_j)}(0)}{k_j!}(a_jx_i)^{k_j}\Big)
    =\sum_{k_1,\dots,k_n=0}^\infty \bigg[\prod_{l=1}^n \frac{g^{(k_l)}(0)}{k_l!}a_l^{k_l}\bigg]\det\Big(x_i^{k_j}\Big)\nonumber \\
&\stackrel{^{(2)}}{=}&\sum_{k_1>\cdots> k_n=0}^\infty \sum_{\sigma\in S_n} \bigg[\prod_{l=1}^n \frac{g^{(k_{\sigma(l)})}(0)}{k_{\sigma(l)}!}a_l^{k_{\sigma(l)}}\bigg]\det\Big(x_i^{k_{\sigma(j)}}\Big)\nonumber\\
    &\stackrel{^{(3)}}{=}&\sum_{k_1>\cdots> k_n=0}^\infty \sum_{\sigma\in S_n} \bigg[\prod_{l=1}^n \frac{g^{(k_l)}(0)}{k_l!}a_l^{k_{\sigma(l)}}\bigg]\det\Big(x_i^{k_j}\Big)\text{sgn}(\sigma)\nonumber\\
    &\stackrel{^{(4)}}{=}&\sum_{k_1>\cdots> k_n=0}^\infty  \bigg[\prod_{l=1}^n \frac{g^{(k_l)}(0)}{k_l!}\bigg]\det\Big(a_i^{k_j}\Big)\det\Big(x_i^{k_j}\Big)\nonumber \\
    &\stackrel{^{(5)}}{=}&\sum_{\substack{ \lambda\\
    l(\lambda)\leq n}}  \bigg[\prod_{l=1}^n \frac{g^{(\lambda_l+n-l)}(0)}{(\lambda_l+n-l)!}\bigg]\det\Big(a_i^{\lambda_j+n-j}\Big)\det\Big(x_i^{\lambda_j+n-j}\Big)\nonumber \\
    &\stackrel{^{(6)}}{=}&|V_{x_{1},\dots,x_n}||V_{a_{1},\dots,a_n}|\sum_{\lambda}\frac{G_\lambda}{C_\lambda}s_{\lambda}(a_{1},\dots,a_n)s_{\lambda}(x_{1},\dots,x_n),
\end{eqnarray*}
where, on the indicated equalities, we have used the following facts.
\begin{enumerate}
    \item[(1)] 
    We used the letters 
 $k_1,\dots,k_n$ 
  to represent the non-negative integers specifying the degree of the terms in the McLaurin expansion for each of the $n$ columns. The sum on the variables $k_1,\dots,k_n$ of the determinants follows from the application of usual properties of determinants.

    \item[(2)] 
   Considering \(\det\Big(x_i^{k_j}\Big)\), it becomes evident that any term in the sum where \(k_l = k_m\) for \(m, l = 1, \dots, n\) with \(l \neq m\) makes no contribution. Therefore, the summation over \(n\)-tuples \((k_1, \dots, k_n)\) reduces to a summation over tuples with distinct entries. Consequently, we can decompose the sum into two parts: the sum over ordered sequences \(k_1 > \cdots > k_n\) and the sum over all permutations of the elements of each sequence, expressed as \((k_{\sigma(1)}, k_{\sigma(2)}, \dots, k_{\sigma(n)})\), where  \(\sigma \) belongs to  $S_n$, the permutation group of $n$ elements.

    \item[(3)] Be aware that, for $\sigma\in S_n$, we have 
    \begin{equation*}
      \prod_{l=1}^n \frac{g^{(k_{\sigma(l)})}(0)}{k_{\sigma(l)}!}= \prod_{l=1}^n \frac{g^{(k_l)}(0)}{k_l!}\quad \text{ and }\quad \det\Big(x_i^{k_{\sigma(j)}}\Big)=\det\Big(x_i^{k_j}\Big)\text{sgn}(\sigma),
    \end{equation*}
    where $sgn(\sigma)$ denotes  the signature of the permutation $\sigma$, taking the value $+1$ if  $\sigma$ is even and $-1$ if  $\sigma$ is odd.
    \item[(4)] Note that
    \begin{equation*}
        \sum_{\sigma\in S_n} \prod_{l=1}^n a_l^{k_{\sigma(l)}}\text{sgn}(\sigma)=\det\Big(a_i^{k_j}\Big).
    \end{equation*}
    \item[(5)] 
After performing the substitution  
\begin{equation*}
    k_l = \lambda_l + n - l, \quad l = 1, \dots, n,
\end{equation*}  
each sequence \(k_1 > \cdots > k_n\) is transformed into a partition with at most \(n\) parts, \((\lambda_1, \dots, \lambda_n)\). It is straightforward to observe that summing over all sequences \(k_1 > \cdots > k_n\) is equivalent to summing over all partitions \(\lambda\) with \(l(\lambda) \leq n\).  

\item[(6)] We have applied Jacobi's bialternant formula:  
\begin{equation*}
    s_\lambda(x_1, \dots, x_n) =  \frac{\det\Big(x_i^{\lambda_j + n - j}\Big)} {|V_{x_1, \dots, x_n}|}.
\end{equation*}  
 Let us observe that the restriction \(l(\lambda) \leq n\) is unnecessary in the summation, since Schur functions are zero whenever \(l(\lambda) > n\).  

\end{enumerate}
\end{proof}

\begin{remark}

As a first consistency check, observe that any permutation of the variables \((a_{\sigma(1)}, \dots, a_{\sigma(n)})\) or \((x_{\sigma(1)}, \dots, x_{\sigma(n)})\) leaves both sides of \eqref{Cor} invariant.  Moreover, note that the evident symmetry under the exchange \((a_1, \dots, a_n) \leftrightarrow (x_1, \dots, x_n)\) on the LHS of \eqref{Cor} is mirrored on the RHS of \eqref{Cor}. This is because, under such exchange, the collocation matrix \(M_G(X)\) is transposed, leaving its determinant unchanged.

Finally, observe that a function \(g\) uniquely determines both the left-hand side of \eqref{Cor}, through the coefficients \(G_\lambda\), and the right-hand side, via the collocation matrix \(M_G(X)\). Consequently, each analytic function \(g\) gives rise to a distinct Cauchy-type identity. This result highlights the versatility of Theorem \ref{Theorem}, which will be further demonstrated in the   following  illustrative examples.

\end{remark}

\begin{exmp}
    For a single variable $x_1$, and $a_1=1$, equation \eqref{Cor} is just the McLaurin expansion of $g(x)$. Notice that for this case, the partitions with one part, $l(\lambda)=1$, are the only ones which contribute to the sum. Now, for $|\lambda|=m$, we have
    \begin{equation*}
        C_\lambda=m!,\quad G_\lambda=g^{(m)}(0), \quad \det M_G(X)=g(x_1),
    \end{equation*}
    and the McLaurin series of $g$ is recovered.
\end{exmp}
\begin{exmp}
     For  $g(x)=\frac{1}{1-x}$,
    we have $G_\lambda=C_\lambda$, and 
    \[
        M_G(X)=\left(\frac{1}{1-a_jx_i}\right)_{1\leq i,j\leq n}.
    \]
    Applying \eqref{Cor}, we recover the Cauchy identity
    \begin{equation} \label{Cid}
      \sum_\lambda s_\lambda(a_1,\dots,a_n)s_\lambda(x_1,\dots,x_n) = \frac{\det \left(\frac{1}{1-a_jx_i}\right)  }{|V_{x_{1},\dots,x_n}||V_{a_{1},\dots,a_n}|} =\prod_{i,j}\frac{1}{1-a_jx_i},  
    \end{equation}
    where  \eqref{Cid} follows from an equivalent variant of the classic Cauchy determinant identity
    \begin{equation*}
        \det\bigg(\frac{1}{x_i-y_j}\bigg)=\frac{|V_{x_1,\dots,x_n}||V_{y_1,\dots,y_n}|}{\prod_{i,j}(x_i-y_j)}.
    \end{equation*}
\end{exmp}

\begin{exmp}
    If $g(x)$ is a polynomial of degree $m\geq n$, that is, 
    \(
        g(x)=b_0+b_1x+\cdots+b_mx^m,
    \)
   then  \begin{equation*}
        M_G(X)= (P_j(x_i))_{1\leq i,j\leq n}, 
         \end{equation*}
         with
         \begin{equation*} P_j(x)=b_0+b_1a_jx+\cdots+b_ma_j^mx^m.
    \end{equation*}
  We obtain a bounded sum since 
\begin{equation*}
  G_\lambda= \left\{ \begin{array}{lcc} C_\lambda B_\lambda, & \text{ if } & \lambda_1\leq m-n+1, \\ 0, & \text{  if }& \lambda_1> m+1-n,  \end{array} \right. 
\end{equation*}
    with $B_\lambda=\prod_{l=1}^{n}b_{\lambda_l+n-l}.$
Thus,
\begin{equation*}
    \sum_{\substack{\lambda\\ \lambda_1\le m-n+1}} B_\lambda s_\lambda(a_1,\dots,a_n)s_\lambda(x_1,\dots,x_n)=\frac{  \det  (P_j(x_i)) }{|V_{x_{1},\dots,x_n}||V_{a_{1},\dots,a_n}|}. 
\end{equation*}
Note that if $m<n$, we have $G_\lambda=0$, but also $\det(P_j(x_i))=0$, since the polynomials $(P_1,\dots,P_n)$ are not linearly independent. Therefore, the identity holds trivially. 

In the particular case where $b_1=\dots =b_m=1$, for which $B_\lambda=1$, we obtain bounded sums that correspond to finite truncations of the Cauchy identity \eqref{Cid}.
 
     As another particular polynomial case, let us consider the sequence of non-negative integers $\mu_1>\dots>\mu_n$, and
    \(
    g(x)=x^{\mu_1}+x^{\mu_2}+\cdots+x^{\mu_n}.
    \)
In this case we have
\begin{equation}
  G_\lambda= \left\{ \begin{array}{ll} C_\lambda, & \text{ if   } \,\,\, \lambda=(\mu_1-n+1,\,\mu_2-n+2,\dots,\,\mu_n), \\ 0,  &\text{  otherwise. }  \end{array}\right. 
\end{equation}
So, we obtain
    \begin{equation}\label{onepart}
s_{(\mu_1-n+1,\dots,\,\mu_n)}(a_1,\dots,a_n)s_{(\mu_1-n+1,\dots,\,\mu_n)}(x_1,\dots,x_n) =\frac{\det  (P_j(x_i)) }{|V_{x_{1},\dots,x_n}||V_{a_{1},\dots,a_n}|} . 
    \end{equation}
The identity \eqref{onepart} can be directly checked by applying the Jacobi’s bialternant formula and basic properties of determinants. Namely, $\det(A B)=\det(A)\det(B)$, and $\det(A^t)=\det(A)$. 
\end{exmp}
Let us call $P$ the set of integer partitions. We define
\begin{eqnarray*}
&& P_{n,\,even} := \{\lambda \in P  \mid  \lambda_l+n-l \text{ is even for }   l=1,\dots,n \}, \\
&& P_{n,\,odd} := \{\lambda \in P\mid  \lambda_l+n-l \text{ is odd for }  l=1,\dots,n \}.
 \end{eqnarray*}
These subsets of partitions are relevant since for even (odd) functions $g(x)$ only the partitions belonging to $P_{n,\,even}$ ($P_{n,\,odd}$) will contribute to the sum in \eqref{Cor}. \\
\begin{exmp}
     For   
    \(
g(x)=\frac{1}{1-x^2},
    \)
    we have 
    \begin{equation*}
           M_G(X)=\bigg(\frac{1}{1-a^2_jx^2_i}\bigg)_{1\leq i,j\leq n}, 
               \end{equation*}
               and
         \begin{equation*}
                 G_\lambda= \left\{ \begin{array}{ll} C_\lambda,  & \text{if }  \lambda \in P_{n,\,even}, \\ 0, & \text{otherwise}.  \end{array} \right.  
    \end{equation*}
    Applying \eqref{Cor}, we obtain
    \begin{equation}\label{minusx2}
      \sum_{\lambda \in P_{n,\,even}}  s_\lambda(a_1,\dots,a_n)s_\lambda(x_1,\dots,x_n)=\prod_{i<j}(a_i+a_j)(x_i+x_j)\prod_{i,j}\frac{1}{1-a^2_ix^2_j},
    \end{equation}
    where the last equality follows from the Cauchy determinant identity, and the fact that
    \begin{equation}\label{Van2}     |V_{x^2_{1},\dots,x^2_n}|=|V_{x_{1},\dots,x_n}|\prod_{i<j}(x_i+x_j). 
    \end{equation}
\end{exmp}

\begin{exmp}
     For   
    \(g(x)=\frac{1}{1+x^2}\), 
    we have 
    \begin{equation*}
   M_G(X)=\bigg(\frac{1}{1+a^2_jx^2_i}\bigg)_{1\leq i,j\leq n}, 
     \end{equation*}
   and 
     \begin{equation*}
     G_\lambda= \left\{ \begin{array}{ll} (-1)^{\frac{2|\lambda|-n(n-1)}{4}}C_\lambda,  & \text{ if }  \lambda \in P_{n,\,even}, \\ 0, & \text{ otherwise.}  \end{array} \right.
    \end{equation*}
    Applying \eqref{Cor}, and taking into account \eqref{Van2}, we obtain
    \begin{equation}\label{plusx2}
      \sum_{\lambda \in P_{n,\,even}}  (-1)^{\frac{2|\lambda|-n(n-1)}{4}}s_\lambda(a_1,\dots,a_n)s_\lambda(x_1,\dots,x_n)  =\frac{\prod_{i<j}(a_i+a_j)(x_i+x_j)}{\prod_{i,j}(1+a^2_ix^2_j)}.
    \end{equation}
   Note that \eqref{plusx2} can be obtained directly from \eqref{minusx2} by the change $a_k\longrightarrow ia_k$, $k=1,\dots, n$. 
\end{exmp}

\begin{exmp}
     For   
    \(g(x)=e^x\), 
we have 
   \begin{equation*}
      M_G(X)= ( e^{a_jx_i} )_{1\leq i,j\leq n}, 
\end{equation*}
and  $G_\lambda=1$ for all partitions. So, we can write  
   \begin{equation*}    
   \sum_\lambda \frac{1}{C_\lambda} s_\lambda(a_1,\dots,a_n)s_\lambda(x_1,\dots,x_n)=\frac{\det ( e^{a_jx_i} )}{|V_{x_{1},\dots,x_n}||V_{a_{1},\dots,a_n}|}. 
\end{equation*}
Furthermore,  if we consider  
    \(g(x)=\sinh(x)\), 
we have
  \begin{equation*}
   M_G(X)=( \sinh(a_jx_i))_{1\leq i,j\leq n},
   \end{equation*}
with
  \begin{equation*}
  G_\lambda= \left\{ \begin{array}{cc} 1,  & \text{ if }  \lambda \in P_{n,\,odd}, \\ 0, & \text{otherwise.}  \end{array} \right.
   \end{equation*}
Then, we have:
   \begin{equation*}
  \sum_{\lambda\in P_{n,\,odd }}\frac{1}{C_\lambda} s_\lambda(a_1,\dots,a_n)s_\lambda(x_1,\dots,x_n)=\frac{ \det( \sinh(a_jx_i)) }{|V_{x_{1},\dots,x_n}||V_{a_{1},\dots,a_n}|}. 
\end{equation*}

\end{exmp} 

\begin{exmp}
     For  
    \(g(x)=\sin(x)\),
we have 
 \begin{equation*}
 M_G(X)=( \sin(a_jx_i))_{1\leq i,j\leq n},
 \end{equation*}
 with
 \begin{equation*}  G_\lambda= \left\{ \begin{array}{cc} (-1)^{\frac{2|\lambda|+n(n+1)}{4}},  & \text{ if }  \lambda \in P_{n,\,odd}, \\ 0, & \text{otherwise,}  \end{array} \right. 
  \end{equation*}
  and the following Cauchy-type identity is derived
    \begin{equation*}
 \sum_{\lambda\in P_{n,\,odd }}\frac{ (-1)^{\frac{2|\lambda|+n(n+1)}{4}}}{C_\lambda} s_\lambda(a_1,\dots,a_n)s_\lambda(x_1,\dots,x_n)\nonumber =\frac{\det( \sin(a_jx_i))}{|V_{x_{1},\dots,x_n}||V_{a_{1},\dots,a_n}|} . 
\end{equation*}

\end{exmp}

\begin{exmp}
For  
\(g(x)=\ln(1-x)\),
we have  
 \[
      M_G(X)=\big(\ln(1-a_jx_i)\big)_{1\leq i,j\leq n} \\
\] 
  with $G_\lambda=C_\lambda/P_\lambda$ and $P_\lambda=\prod_{i=1}^{l(\lambda)}\lambda_i$.  The following identity is obtained:
\[  \sum_\lambda \frac{1}{P_\lambda}s_\lambda(a_1,\dots,a_n)s_\lambda(x_1,\dots,x_n)=\frac{\det (\ln(1-a_jx_i))}{|V_{x_{1},\dots,x_n}||V_{a_{1},\dots,a_n}|}\].

\end{exmp}

\section*{Acknowledgement}
This work was partially supported through the Spanish research grants PID2022-138569NB-I00 and RED2022-134176-T (MCI/AEI) and by Gobierno de Arag\'{o}n (E41$\_$23R, S60$\_$23R).


%



\end{document}